\documentclass[11pt]{amsart}

\usepackage{mathrsfs}
\usepackage{latexsym}
\usepackage{latexsym, url}
\usepackage{amsthm}
\usepackage{amsmath}
\usepackage{amsfonts}
\usepackage{amssymb}
\usepackage[dvips]{graphicx}
\usepackage{xypic}
\input xy
\xyoption{all}

\newtheorem{theorem}{Theorem}[subsection]
\newtheorem{lemma}[theorem]{Lemma}
\newtheorem{proposition}[theorem]{Proposition}
\newtheorem{corollary}[theorem]{Corollary}

\newtheorem{theorem1}{Theorem}[section]
\newtheorem{corollary1}[theorem1]{Corollary}

\theoremstyle{definition}
\newtheorem{definition}[theorem]{Definition}

\newtheorem{remark}[theorem]{Remark}

\newtheorem{example}[theorem]{Example}
\newtheorem{remark1}[theorem1]{Remark}
\newtheorem{example1}[theorem1]{Example}

\newcommand\Iso{\operatorname{Iso}}
\newcommand\op{\operatorname{op}}
\newcommand\Set{\operatorname{\mathcal{S}et}}
\newcommand\Gra{\operatorname{\bf Gra}}
\newcommand\SSet{\operatorname{\mathcal{SS}et}}
\newcommand\C{\operatorname{\bf C}}
\newcommand\Ho{\operatorname{Ho}}
\newcommand\colim{\operatorname{colim}}
\newcommand\hocolim{\operatorname{hocolim}}

\newcommand\ca{\mathcal {A}}
\newcommand\cb{\mathcal {B}}
\newcommand\cc{\mathcal {C}}
\newcommand\cd{\mathcal {D}}
\newcommand\cu{\mathcal {U}}
\newcommand\ci{\mathcal {I}}
\newcommand\cj{\mathcal {J}}
\newcommand\ck{\mathcal {K}}
\newcommand\cn{\mathcal {N}}
\newcommand\cl{\mathcal {L}}
\newcommand\cm{\mathcal {M}}
\newcommand\crr{\mathcal {R}}

\newcommand\cw{\mathcal {W}}
\newcommand\cx{\mathcal {X}}

\begin{document}
\title{Small presentations of model categories and Vop\v{e}nka's principle}
\author{G. Raptis}                
\author{J. Rosick\'{y}}
       
\address{
\newline G. Raptis \newline
Fakult\"{a}t f\"{u}r Mathematik, Universit\"{a}t Regensburg, 93040 Regensburg, Germany}
\email{georgios.raptis@ur.de}

\address{
\newline J. Rosick\'{y}\newline
Department of Mathematics and Statistics, Faculty of Sciences, Masaryk University  \newline
Kotl\'{a}\v{r}sk\'{a} 2, 611 37 Brno, Czech Republic}
\email{rosicky@math.muni.cz}


\begin{abstract}
We prove existence results for small presentations of model categories generalizing a theorem
of D. Dugger from combinatorial model categories to more general model categories. Some of these    
results are shown under the assumption of Vop\v{e}nka's principle. Our main theorem applies in particular to 
cofibrantly generated model categories where the domains of the generating cofibrations satisfy a slightly stronger 
smallness condition. As a consequence, assuming Vop\v{e}nka's principle, such a cofibrantly generated model category 
is Quillen equivalent to a combinatorial model category. Moreover, if there are generating sets which consist of presentable 
objects, then the same conclusion holds without the assumption of Vop\v{e}nka's principle. We also correct a mistake from 
previous work that made similar claims. 
\end{abstract}

\maketitle
 
\section{Introduction}

The authors addressed in previous articles \cite{Ra, R2} the question of whether a cofibrantly generated model category is Quillen equivalent to a combinatorial model category 
and connected this question with special set-theoretical assumptions. More specifically, the first author claimed in \cite{Ra} that assuming Vop\v enka's principle any cofibrantly 
generated model category is Quillen equivalent to a combinatorial one. The proof appeared to generalize also to a weaker notion of a cofibrantly generated model category where 
the domains of the generating sets need not be small. The second author claimed in \cite{R2} that Vop\v{e}nka's principle is equivalent to the statement that such a generalized 
cofibrantly generated model category admits a combinatorial model up to Quillen equivalence. Unfortunately, the proof in \cite{Ra} contains a mistake which cannot be repaired 
without adding further assumptions. 

The purpose of the present article is to point out this mistake and present a different approach to the problem of finding combinatorial models for cofibrantly generated model 
categories. Our main theorem establishes a slightly weaker version of the claim in \cite{Ra} but it applies to more general model categories as well. Moreover, we show that there 
are examples of cofibrantly generated model categories in the generalized sense of \cite{R2}, which do not admit a combinatorial model, independently of Vop\v{e}nka's principle. 

Our main results are based on a delicate analysis of the properties of weak factorization systems that are generated by sets of morphisms and which satisfy various kinds of smallness conditions. We introduce a hierarchy of smallness conditions, as well as associated notions of cofibrantly generated model categories, and investigate their implications for the 
categories of cofibrant objects and for the classes of weak equivalences. Recall that a weak factorization system $(\cl,\crr)$ in a cocomplete category $\cm$ is cofibrantly generated if 
there is a set of morphisms $\cx \subseteq \cl$ such that 
\begin{enumerate}
\item[(a)] the domains of $\cx$ are small relative to $\cx$,
\item[(b)] $\cx^\square=\crr$.
\end{enumerate}
This corresponds to the standard notion that appears in the definition of a cofibrantly generated model category. We introduce the following variants of this definition (see Definition 
\ref{many-def} for a more precise definition and a more complete list):
\begin{enumerate} 
 \item[(1)] $(\cl, \crr)$ is \emph{weakly} cofibrantly generated if there is a set $\cx \subseteq \cl$ such that $\cl = \mathrm{cof}(\cx)$ where $\mathrm{cof}(\cx)$ denotes the smallest 
 class of morphisms in $\cm$ which contains $\cx$ and is closed under pushouts, transfinite compositions, and retracts. 
 \item[(2)] $(\cl, \crr)$ is \emph{strictly} cofibrantly generated if there is a set $\cx \subseteq \cl$ that satisfies (b) and the domains of the morphisms in $\cx$ are small with 
 respect to $\lambda$-directed good colimits whose links are $\cx$-cellular morphisms. 
 \item[(3)] $(\cl, \crr)$ is \emph{perfectly} cofibrantly generated if there is a set $\cx \subseteq \cl$ that satisfies (b) and the domains and the codomains of the morphisms in 
 $\cx$ are presentable objects in $\cm$. 
\end{enumerate}
The notion of a good colimit will be recalled in Section 2. The smallness condition in (2) appears to be stronger than the one in 
(a) above, at least when the corresponding smallness rank is 
uncountable. Roughly speaking, while the smallness condition in (a) is concerned with directed chains of $\cx$-cellular morphisms, the one in (2) applies also to wider directed diagrams of $\cx$-cellular 
morphisms. This seems to be a natural variant of the standard definition from the viewpoint of the fat small object argument \cite{MRV}. 

\medskip 

Model categories were first introduced and studied by Quillen in his seminal work \cite{Qu}. Following the standard conventions from the more recent treatments of the subject \cite{Hi, Ho}, a model category consists of a complete and cocomplete category $\cm$ together with three classes of morphisms $\mathcal{C}of$, $\cw$ and $\mathcal{F}ib$ such that $\cw$ contains the isomorphisms and has the 2-out-of-3
property and the two pairs $(\mathcal{C}of \cap \cw, \mathcal{F}ib)$ and $(\mathcal{C}of, \cw \cap \mathcal{F}ib)$ define weak factorization systems with functorial factorizations. A model
category $(\cm, \mathcal{C}of, \cw, \mathcal{F}ib)$ is cofibrantly generated if both weak factorization systems are cofibrantly generated. Moreover, $(\cm, \mathcal{C}of, \cw, \mathcal{F}ib)$ is a combinatorial model category if it is cofibrantly generated and its underlying 
category $\cm$ is locally presentable. Examples of combinatorial model 
categories include the standard model category of simplicial sets $\SSet$, the projective model category of simplicial presheaves $\SSet^{\mathcal{C}^{\op}}$ where $\mathcal{C}$ is a small category, and the left Bousfield localization $L_S \SSet^{\mathcal{C}^{\op}}$
of $\SSet^{\mathcal{C}^{\op}}$ at a set of morphisms $S$. 

A small presentation of a model category $(\cm, \mathcal{C}of, \cw, \mathcal{F}ib)$ consists of a small category $\mathcal{C}$, a set of morphisms $S$ in $\SSet^{\mathcal{C}^{\op}}$ and a Quillen equivalence 
$$F \colon L_S \SSet^{\mathcal{C}^{\op}} \rightleftarrows \cm \colon G.$$
This notion was introduced by Dugger \cite{D1} who showed that every combinatorial model category admits a small presentation. In this paper, we extend this result to more general model categories and in this way we conclude that these model categories admit combinatorial models. Our main results for the existence of small presentations apply under the 
assumption that $(\mathcal{C}of, \cw \cap \mathcal{F}ib)$ is \emph{strictly} cofibrantly generated. 
We do not know an example of a cofibrantly generated model category for which $(\mathcal{C}of, \cw \cap \mathcal{F}ib)$ fails to be strictly cofibrantly generated - such an example cannot be finitely generated. 

More specifically, in Section \ref{section4}, we obtain the following results (see Theorem \ref{main}). Let $(\cm, \mathcal{C}of, \cw, \mathcal{F}ib)$ be a model category. 
\begin{itemize}
 \item[(A)] Suppose that $(\mathcal{C}of, \cw \cap \mathcal{F}ib)$ is strictly cofibrantly generated. Then, assuming Vop\v{e}nka's principle, $\cm$ admits a small presentation. Therefore 
 $\cm$ is Quillen equivalent to a combinatorial model category. 
 \item[(B)] Suppose that both weak factorization systems of $\cm$ are perfectly cofibrantly generated. Then $\cm$ admits a small presentation. Therefore $\cm$ is Quillen equivalent to a 
 combinatorial model category.
\end{itemize}
On the other hand,
\begin{itemize}
 \item[(C)] There are model categories whose weak factorization systems are weakly cofibrantly generated but they are not Quillen equivalent to a combinatorial model category. Moreover, 
 even the homotopy category of such a model category is not equivalent to the homotopy category of a combinatorial model category in general. 
\end{itemize}
The method we use to prove (A) and (B) follows Dugger's method \cite{D1} for proving that a combinatorial model category has a small presentation. Based on the preliminary results about cofibrant objects and 
weak equivalences in Section 3, we show that this method can be applied to more general model categories - assuming Vop\v{e}nka's principle or not. Concerning the role of the large cardinal axiom in Vop\v{e}nka's principle, it is worth mentioning that this is used through several of its non-trivially equivalent forms or non-trivial implications and it is required at many steps in the proof of (A). 

The results in this paper help to clarify the situation regarding the comparison between cofibrantly generated and combinatorial model categories, but aspects of this comparison and its precise 
set-theoretical status remain open. For example, we do not know if there is an example of a cofibrantly generated model category $(\cm, \mathcal{C}of, \cw, \mathcal{F}ib)$ which does not have a small presentation. Moreover, we do not know if assuming the negation of Vop\v{e}nka's principle, there is such an example for which $(\mathcal{C}of, \cw \cap \mathcal{F}ib)$ is strictly cofibrantly generated.  

\medskip

The paper is organized as follows. Section 2 is concerned with the general properties of weak factorization systems: there we introduce and compare different notions of cofibrant 
generation. Section 3 collects some useful properties  of strictly cofibrantly generated weak factorization systems. In particular, we prove that the category of cofibrant objects 
has a small (colimit-)dense subcategory (see Theorem \ref{swfs-general case}) and explore the consequences of this property. Moreover, we give general 
conditions on a model category for the class of weak equivalences (between cofibrant objects) to be closed under large enough filtered colimits (Proposition \ref{weak-eq}). In Section 4, we first present a short review of Dugger's method \cite{D1} for the 
construction of small presentations as it applies to more general model categories. Then we deduce our main results 
about the existence of small presentations (Theorem \ref{main} and Corollary \ref{main2}). Lastly, in Subsection \ref{counterexample}, we discuss an example of a weakly cofibrantly 
generated model category whose homotopy category is not equivalent to the homotopy category of a combinatorial model category. Finally, in Section 5, we discuss the mistake in the argument of \cite{Ra} and correct it under additional assumptions.

\

\noindent \textit{Acknowledgements.} The first author was partially supported by SFB 1085 - \emph{Higher Invariants}, University of Regensburg. The second author 
was supported by the Grant Agency of the Czech Republic under the grant P201/12/G028.

\section{Generating Classes of Cofibrations}

In this section we review the standard definition of a cofibrantly generated weak factorization system and introduce some variations of this notion. We are mainly interested in a slightly stronger notion 
of a weak factorization system where the domains of the generating cofibrations satisfy a stronger smallness condition. This condition will be required in the proofs of our main 
results but it may also be of independent interest. It is suggested by the fat small object argument \cite{MRV} in the same way that the standard smallness condition
arises naturally from Quillen's small object argument. 

\subsection{Small objects}

\medskip

Let $\cm$ be a cocomplete category and $\cx$ a class of morphisms. The class $\mathrm{cell}(\cx)$ of $\cx$-\textit{cellular morphisms} is the smallest class of morphisms 
which contains $\cx$ and is closed under transfinite compositions and pushouts along morphisms in $\cx$. The class $\mathrm{cof}(\cx)$ of $\cx$-\textit{cofibrations} is 
the smallest class of morphisms which is in addition closed under retracts. Let $\cx^{\square}$ (resp. ${}^\square\cx$) denote the class of morphisms which have the right 
lifting property (resp. left lifting property) 
with respect to each morphism in $\cx$. It is easy to see that every $\cx$-cofibration has the left lifting property with respect to $\cx^{\square}$, that is,
$\mathrm{cof}(\cx) \subseteq {}^{\square}(\cx^{\square})$. 

\medskip

Given a poset $P$ and $x\in P$, $\downarrow x=\{y\in P\mid y\leq x\}$ denotes the \emph{initial segment}
associated with $x$. An element $x$ of a poset $P$ is 
\textit{isolated} if the subposet
$$
\downdownarrows x=\{y\in P\mid y< x\}
$$ 
has a greatest element $x^-$ which is called the \textit{predecessor} of $x$. A non-isolated element, which is not a least element, is called a \textit{limit element}. 
For a poset $P$, a diagram $D\colon P\to\cm$ is called \textit{smooth} if for every limit element $x\in P$, the diagram 
$$(D(y \to x)\colon Dy \to Dx)_{y<x}$$ 
defines a colimit cocone of the restriction of $D$ to the subposet $\downdownarrows x$.

A poset $P$ is \emph{well-founded} if each nonempty subset of $P$ contains a minimal element. A poset $P$ is called \textit{good} if it is well-founded and has a least element $\perp$. 
Note that every nonempty well-ordered set is good. For a regular cardinal $\lambda$, we say that a good poset $P$ is $\lambda$-\textit{good} if the initial segment 
$\downarrow x$ has cardinality $<\lambda$ for each $x \in P$.  A \textit{good} diagram $D\colon P\to\cm$ is a smooth 
diagram where $P$ is a good poset. Such a diagram is essentially determined by $D(\perp)$ and the morphisms $D(x^- \to x)$ 
for all isolated elements $x \in P$. These morphisms $D(x^- \to x): D(x^-) \to D(x)$ are the \textit{links} of the diagram $D$. The \emph{composite} of a good diagram $D : P \to \cm$ is the canonical morphism $D(\perp) \to \colim_P D$. 

\medskip 

An object $X \in \cm$ is $\lambda$-\textit{small relative to} $\cx$ if $\cm(X,-)$ preserves colimits of smooth $\lambda$-directed chains of $\cx$-cellular morphisms. 
We say that $X$ is \textit{small relative to} $\cx$ if it is $\lambda$-small relative to $\cx$ for some regular cardinal $\lambda$. This is the standard notion of 
smallness that appears in the definition of cofibrantly generated model categories and it is suggested by the small object argument. This argument was originally used 
by Quillen in \cite{Qu} and it was later formalized in the following way.  

\begin{theorem}[Small object argument]
Let $\cm$ be a cocomplete category and $\cx$ a set of morphisms. Suppose that the domains of the morphisms in $\cx$ are small relative to $\cx$. Then every morphism $f$ in $\cm$ admits 
a functorial factoriation $f= p i$ into a $\cx$-cellular morphism $i \in \mathrm{cell}(\cx)$ followed by a morphism $p \in \cx^{\square}$.
\end{theorem}
\begin{proof}
See \cite[Theorem 2.1.14]{Ho}. 
\end{proof}

\begin{remark}
If the domains of the morphisms in (the set) $\cx$ are small relative to $\cx$, then the small object argument together with the retract argument imply that
$\mathrm{cof}(\cx) = {}^{\square}(\cx^{\square})$ (see \cite[2.1.15]{Ho}). 
\end{remark}

We will consider the following variation of the notion of smallness. 

\begin{definition} \label{strictly-small-def}
Let $\cm$ be a category, $\cx$ a set of morphisms, and $\lambda$ a regular cardinal. An object $X$ in $\cm$ is called 
\textit{strictly $\lambda$-small relative to} $\cx$ if 
$\cm(X,-)$ preserves $\lambda$-directed good colimits whose links are $\cx$-cellular morphisms. We say that $X$ is \textit{strictly small relative to} $\cx$ if it is 
strictly $\lambda$-small for some regular cardinal $\lambda$. 
\end{definition}

\begin{remark}
In Definition \ref{strictly-small-def}, it is possible to restrict to $\lambda$-directed good colimits whose links are obtained as pushouts of morphisms in $\cx$ as this 
leads to an equivalent notion. The equivalence can be demonstrated by replacing the $\cx$-cellular links of a diagram by directed chains whose links are pushouts along 
morphisms in $\cx$. An analogous  statement is also true and well known for the definition of small objects. 
\end{remark}

We note that the composite of a good diagram whose links are $\cx$-cellular morphisms has the left lifting property with respect to $\cx^{\square}$ (see the proof of 
\cite[Proposition 4.4]{MRV} or \cite[Lemma A.1.5.6]{Lu}). A strictly $\lambda$-small object is clearly $\lambda$-small. Therefore, a set of morphisms $\cx$ whose domains are strictly small with respect to $\cx$ permits the small object argument. 
Moreover, it also permits the fat small object argument \cite{MRV} as stated below (see also \cite[A.1.5]{Lu}). 

We first recall some notation from \cite{MRV}. Let $\mathrm{Po}(\cx)$ denote the class of morphisms which are obtained as pushouts along a morphism in $\cx$. Let 
$\mathrm{GdDirPo}_{\lambda}(\cx)$ denote the class of morphisms which are obtained as composites of $\lambda$-good $\lambda$-directed diagrams whose links 
are in $\mathrm{Po}(\cx)$. 

\begin{theorem}[Fat small object argument] \label{fsoa}
Let $\cm$ be a cocomplete category, $\cx$ a set of morphisms, and $\lambda$ a regular cardinal. Suppose that the domains of the morphisms in $\cx$ are strictly $\lambda$-small relative to 
$\cx$. Then every morphism $f$ in $\cm$ admits a factorization $f = p i$ into a morphism $i \in \mathrm{GdDirPo}_{\lambda}(\cx)$ followed by a morphism $p \in \cx^{\square}$.
\end{theorem}
\begin{proof}
This factorization can be obtained using a variation of the small object argument as in \cite[Theorem A.1]{MRV}. Alternatively, the small object argument yields a functorial factorization 
$f=pi$ where $i$ is a $\cx$-cellular morphism and $p \in \cx^{\square}$. Then the proof of \cite[Theorem 4.11]{MRV} applies to show that $i \in \mathrm{GdDirPo}_{\lambda}(\cx)$ 
as required (see also \cite[Remark A.2]{MRV}). 
\end{proof}

\begin{remark}\label{cofinal}
Any directed poset $P$ contains a cofinal good subposet $Q$ (see, e.g., \cite[2.7.2]{F}). If $P$ is $\lambda$-directed, then $Q$
is $\lambda$-directed too. But, given a $\lambda$-directed diagram $P\to\cm$, its restriction to $Q$ does not need to be smooth. 
So we do not know whether, for a strictly $\lambda$-small object $X$, the functor $\cm(X,-)$ preserves $\lambda$-directed colimits 
of $\cx$-cellular morphisms in general.
\end{remark}

\subsection{Weak factorization systems}

\medskip

We recall that a \textit{weak factorization system} $(\cl,\crr)$ in a category $\cm$ consists of two classes of morphisms $\cl$ and $\crr$  in $\cm$ such that
\begin{enumerate}
\item[(1)] $\crr = \cl^{\square}$, $\cl = {}^\square \crr$, 
\item[(2)] every morphism $h$ of $\cm$ has a factorization $h=gf$ with
$f\in \cl$ and $g\in \crr$.
\end{enumerate}
The small object argument provides examples of weak factorization systems in general cocomplete categories. One feature of these weak factorization systems is that they are 
determined by a set of morphisms. 

\begin{definition}
Let $(\cl, \crr)$ be a weak factorization system in $\cm$. A set of morphisms $\cx \subseteq \cl$ is called a \textit{generator for} $\crr$ if $\cx^\square=\crr$.
\end{definition} 

This property isolates one of the requirements for the weak factorization systems in the definition of a cofibrantly generated model category. A generator 
$\cx$ for $\crr$ obviously determines $\crr$, and therefore $\cl$, in terms of lifting properties. As the following remark shows, this notion generalizes 
the standard notion of a strong generator in a category. 

\begin{remark} \label{strong-generators}
Let $\cm$ be a category with finite coproducts and equalizers and consider the trivial (weak) factorization system 
$(\cm^{\to}, \Iso_{\cm})$ which is defined by the classes of all morphisms and isomorphisms in 
$\cm$, respectively. We claim that this weak factorization system has a generator if and only if $\cm$ has a strong generator. 

Suppose that $\ca$ is a strong generator of $\cm$. Under the assumption on equalizers, this means that a morphism 
$f: X \to Y$ is an isomorphism if $\cm(A, f): \cm(A, X) \to \cm(A, Y)$ is a bijection for all $A \in \ca$. 
Consider the set $\cx$ which consists of the morphisms 
$$0_A: 0\to A \ \ \text{and} \ \ \nabla_A: A \sqcup A\to A$$ 
for all $A\in\ca$ where $0$ denotes an initial object and $\nabla_A$ is the codiagonal. 
A morphism $f: X \to Y$ has the right lifting property with respect to $0_A$ (resp. $\nabla_A)$ if and only if 
the map 
$$\cm(A, f): \cm(A, X) \to \cm(A, Y)$$ 
is surjective (resp. injective). Hence $f \in \cx^{\square}$ if and only if $\cm(A, f)$ is bijective for all $A \in \ca$. 

For the converse, suppose that $\cx \subseteq \cm^{\to}$
is a generator for $\Iso_{\cm}$. Consider the set $\ca$ of domains and codomains of the morphisms in $\cx$. It is easy to see that a morphism $f: X \to Y$ 
such that $\cm(A, f)$ is bijective for all $A \in \ca$ also has the right lifting property with respect to $\cx$, and therefore it is an isomorphism. 
\end{remark}

Combined with the notions of smallness above, we have the following hierarchy of notions of cofibrant generation for weak factorization systems. 

\begin{definition} \label{many-def}
Let $(\cl,\crr)$ be a weak factorization system in a cocomplete category $\cm$ and $\lambda$ a regular cardinal. 
\begin{itemize} 
\item[(a)] $(\cl, \crr)$ is called \textit{setwise cofibrantly generated} if there is a generator $\cx \subseteq \cl$ 
for $\crr$.
\item[(b)] $(\cl, \crr)$ is called \textit{weakly cofibrantly generated} if there is a generator 
$\cx \subseteq \cl$ for $\crr$ such that $\cl = \mathrm{cof}(\cx)$. In this case, $\cx$ is called a \textit{generating 
set for} $(\cl, \crr)$. 
\item[(c)] $(\cl, \crr)$ is called ($\lambda$-)\textit{cofibrantly generated} if there is a generator 
$\cx \subseteq \cl$ for $\crr$ such that the domains of the morphisms in $\cx$ are ($\lambda$-)small relative to $\cx$. 
\item[(d)] $(\cl,\crr)$ is called \textit{strictly} ($\lambda$-)\textit{cofibrantly generated} if there is a 
generator $\cx \subseteq \cl$ for $\crr$ such that the domains of the morphisms in $\cx$ are strictly ($\lambda$-)small relative to $\cx$.
\item[(e)] $(\cl, \crr)$ is called \textit{semi-perfectly} ($\lambda$-)\textit{cofibrantly generated} if there is a generator 
$\cx \subseteq \cl$ for $\crr$ such that the domains of the morphisms in $\cx$ are ($\lambda$-)presentable objects in $\cm$.
\item[(f)] $(\cl, \crr)$ is called \textit{perfectly} ($\lambda$-)\textit{cofibrantly generated} if there is a generator 
$\cx \subseteq \cl$ for $\crr$ such that the domains and the codomains of the morphisms in $\cx$ are ($\lambda$-)presentable objects in $\cm$.
\end{itemize}
\end{definition}

The definition in (c) corresponds to the standard notion that appears in the definition of cofibrantly generated model categories. 
The weaker notion in (b) was considered in \cite{R2} together with the corresponding 
weaker notion of a cofibrantly generated model category. Concerning the definition in (a), note that a generator $\cx$ does 
not permit the small object argument in general, and therefore we cannot identify $\cl$ with $\mathrm{cof}(\cx)$. We are mainly 
interested in (d) which is motivated by the fat small object argument as previously explained. We have the following obvious 
implications:
$$\xymatrix{
\text{perfectly cof. generated} \ar@{=>}[r] & \text{strictly cof. generated} \ar@{=>}[r] & \text{cofibrantly generated} \ar@{=>}[d] \\
&& \text{weakly cof. generated} \ar@{=>}[d] \\
&& \text{setwise cof. generated}
}
$$

\begin{example} \label{compact-hausdorff}
Let $\ck$ be the category of compact Hausdorff spaces with the trivial weak factorization system $(\ck^{\to}, \Iso_{\ck})$. Since the one-point space $\ast$ is a strong generator in $\ck$, 
this weak factorization system has a generator $\cx$ which consists of the maps (see Remark \ref{strong-generators})
$$0: \varnothing \to \{*\} \ \ \text{and} \ \ \nabla: \{*\} \sqcup \{*\} \to \{*\}.$$
Moreover, it is weakly cofibrantly generated because every morphism in $\ck$ is in $\mathrm{cof}(\cx)$. To see this, note first that every surjective map can 
be obtained as a pushout along a coproduct of copies of the map $\nabla$. For an injective map $f: X \to Y$, consider the factorization  
$$X \stackrel{f}{\cong} f(X) \stackrel{j}{\longrightarrow} f(X) \sqcup \mathbb{F}(Y - f(X)^{\delta}) \stackrel{p}{\longrightarrow} Y$$
where $(Y-f(X))^{\delta}$ denotes the discrete topological space with underlying set $Y - f(X)$ and $\mathbb{F}$ is the associated free compact Hausdorff space. Recall that for a set $Z$,  $\mathbb{F}(Z)$ is the compact Hausdorff space 
of all ultrafilters on $Z$.  The map $j$ is clearly in $\mathrm{cell}(\cx)$, by definition. The map $p$ is surjective, so it is in $\mathrm{cof}(\cx)$, too. Note that in this case we have 
$\mathrm{cell}(\cx) = \mathrm{cof}(\cx)$.  

However, this weak factorization system is not cofibrantly generated because $\varnothing$ is the only small object in $\ck$. To see this, it suffices to show that $*$ is not small in $\ck$. A regular cardinal $\lambda$ can be written as the colimit of the 
$\lambda$-directed chain of all ordinals with cardinality less than $\lambda$. Then an ultrafilter on $\lambda$ which contains the filter of all subsets whose complement has cardinality less than $\lambda$ defines a point in $\mathbb{F}(\lambda)$ which does not factor through any stage of the chain. Hence, $*$ is not $\lambda$-small in $\ck$. 
\end{example}

We do not know an example of a cofibrantly generated weak factorization system which is not strictly cofibrantly generated. Such an example does not exist in the case of $\aleph_0$ 
by the following result. 

\begin{proposition}\label{finitely small case}
Let $(\cl,\crr)$ be a cofibrantly generated weak factorization system in $\cm$ generated by a set of morphisms $\cx$ whose domains are $\aleph_0$-small relative to 
$\cx$. Then $(\cl,\crr)$ is strictly $\aleph_0$-cofibrantly generated. 
\end{proposition}
\begin{proof}
It suffices to show that any domain $X$ of a morphism from $\cx$ is strictly $\aleph_0$-small. Let $P$ be a directed 
good poset and $D:P \to \cm$ a good diagram whose links are $\cx$-cellular morphisms. Following \cite[1.6]{AR}, we can 
express $P$ as a union of a smooth chain of directed subposets $P_k$ each of which has cardinality less than that of $P$. 
Let $D_k$ denote the restriction of $D$ to $\downarrow P_k$.  Here $\downarrow P_k$ denotes the initial segment generated by $P_k$. Note that $P_k$ is cofinal in $\downarrow P_k$. 
Moreover, $D_k$ is again a good diagram whose links are in $\mathrm{cell}(\cx)$. 
Following \cite[Remark 4.8]{MRV}, the induced morphisms 
for $k < k'$,
$$\colim D_{k} \to \colim D_{k'}$$ 
are $\cx$-cellular. 
Clearly, we have $\colim_k \colim_{\downarrow P_k} D_k \cong \colim_P D $. Then we can proceed by transfinite induction on the 
cofinality of $P$, that is, the smallest cardinality for which there is a cofinal subposet (cf. \cite[Corollary 1.7]{AR}).
\end{proof}

\section{Cofibrant Objects and Weak Equivalences}

Let $\cm$ be a cocomplete category and $(\cl, \crr)$ a weak factorization system in $\cm$. An object $X$ in $\cm$ is called $\cl$-\textit{cofibrant}, or simply \textit{cofibrant}, 
if the morphism $0 \to X$ is in $\cl$, where $0$ denotes an initial object. Let $\cm_c  \subseteq \cm$ denote the full subcategory which is spanned by the cofibrant objects.

In this section, we discuss some useful properties of the full subcategory $\cm_c  \subseteq \cm$ in the case where $(\cl, \crr)$ is strictly cofibrantly 
generated. We will also use these properties to identify general conditions on a model category so that the class of weak equivalences (between cofibrant objects) is closed 
under large enough filtered colimits. The results in this section form the technical backbone of the proof of our main theorem in the next section (Theorem \ref{main}). 

\subsection{Small dense subcategories of cofibrant objects} 

\medskip

Let $\cm$ be a cocomplete category and $(\cl, \crr)$ a strictly cofibrantly generated weak factorization system in $\cm$ with a generating set $\cx$. 
For any cardinal $\lambda$, let $\widetilde{\ca}(\cx, \lambda)$ be the smallest full subcategory of $\cm$ which contains the domains and codomains of the morphisms in 
$\cx$ and is closed under $\lambda$-small colimits in $\cm$. Here $\lambda$-small refers to colimits of diagrams indexed by categories whose sets of morphisms have cardinality smaller than $\lambda$. We consider the small full 
subcategory $\ca(\cx, \lambda)$ of $\cm_c$ given by 
$$\ca(\cx, \lambda) = \widetilde{\ca}(\cx, \lambda) \cap \cm_c.$$
Note that the construction of $\ca(\cx, \lambda)$ has the following properties:
\smallskip
\begin{itemize}
\item[(a)] if $\lambda \leq \mu$, then $\ca(\cx, \lambda) \subseteq \ca(\cx, \mu)$.
\item[(b)] if $\lambda = \bigvee_{i < \alpha} \lambda_i$, then $\ca(\cx, \lambda) = \bigcup_{i < \alpha} \ca(\cx, \lambda_i)$.
\end{itemize}

\smallskip

In the following statements, we will make use of the definition and the properties of preaccessible categories (see \cite{AR1}). For a regular cardinal $\lambda$, a category $\cc$ is called $\lambda$-\textit{preaccessible} 
if it has a set of $\lambda$-presentable objects $\ca$ such that each object in $\cc$ is a $\lambda$-filtered colimit of objects in $\ca$. Note that we do not assume that 
$\cc$ has all $\lambda$-filtered colimits. We say that $\cc$ is \textit{preaccessible} if it is $\lambda$-preaccessible for some regular cardinal $\lambda$. A preaccessible category has a small dense subcategory \cite{AR1}. 
Conversely, assuming Vop\v{e}nka's principle, a category which has a small dense subcategory is preaccessible  \cite[Theorem 1]{AR1}. 

Recall also that a full subcategory $\ca$ of a category $\cc$ is \textit{colimit-dense} provided that every object of $\cc$ is a colimit of a diagram in $\ca$. 
Assuming Vop\v{e}nka's principle, a category which has a small colimit-dense subcategory has also a small dense subcategory \cite[Theorem 6.35]{AR}.

\begin{theorem}\label{swfs-general case} 
Let $\cm$ be a cocomplete category and $(\cl, \crr)$ a strictly $\kappa$-cofibrantly generated weak factorization system in $\cm$ with generating set $\cx$. 
\begin{itemize}
\item[(1)] For every regular cardinal $\lambda \geq \kappa$,  every object of $\cm_c$ is a $\lambda$-filtered colimit in $\cm$ of objects in $\ca(\cx, \lambda)$.
\item[(2)]  Suppose that $(\cl, \crr)$ is perfectly $\kappa$-cofibrantly generated with generating set $\cx$ and let $\lambda \geq \kappa$ be a regular cardinal. Then the small colimit-dense subcategory $\ca(\cx, \lambda) \subseteq \cm_c$ 
consists of $\lambda$-presentable objects in $\cm$. In particular, every cofibrant object is presentable in $\cm$. 
\item[(3)] Assuming Vop\v enka's principle, $\cm_c$ is preaccessible and therefore every cofibrant object is presentable in $\cm_c$. Moreover, there are arbitrarily large regular cardinals $\lambda$ such that every object of $\ca(\cx, \lambda)$ is $\lambda$-presentable in $\cm_c$. 
\end{itemize}
\end{theorem}
\begin{proof}
First note that the domains of the morphisms in $\cx$ are strictly $\lambda$-small for each $\lambda \geq \kappa$. Following the proof of \cite[Theorem 4.11]{MRV} (cf. Theorem \ref{fsoa}), every $\cx$-cellular object 
is a $\lambda$-good $\lambda$-directed colimit of a diagram in $\cm$ consisting of objects in $\ca(\cx, \lambda)$ and 
$\cx$-cellular morphisms. Then every cofibrant object is a $\lambda$-directed colimit in $\cm$ of objects from $\ca(\cx, \lambda)$ (see the proof of \cite[Corollary 5.1]{MRV} and 
\cite[Proposition 2.3.11]{MP}). This completes the proof of (1). (2) follows from the construction of the small colimit-dense subcategory $\ca(\cx, \lambda)$. The presentability of the 
cofibrant objects 
can also be seen directly from the small object argument. 

The first part of (3) follows from (1) and \cite[Theorem 6.35]{AR} combined with \cite[Theorem 1]{AR1} or \cite[Corollary 6.11]{AR}. For the second part, we may assume that 
$\cm_c$ is a full subcategory of the category of graphs $\Gra$ because $\cm_c$ has a small dense subcategory (see \cite[Proposition 1.26, Theorem 2.65]{AR}). Assuming Vop\v{e}nka's principle, 
the inclusion $\cm_c \to \Gra$ preserves $\lambda_0$-directed colimits for some $\lambda_0$ by \cite[Theorem 6.9]{AR}.  For $X\in\cm_c$, $|X|$ denotes the cardinality of the underlying set of the graph $X$. Then, given a regular cardinal $\alpha \geq \lambda_0$,  $|X|<\alpha$ implies that $X$ is 
$\alpha$-presentable in $\cm_c$ because it is so in $\Gra$. For \emph{any} cardinal $\alpha$, let $\alpha^* \geq \alpha$ be the smallest cardinal  such that for each $X \in \ca(\cx, \alpha)$, we have that $|X| < \alpha^*$. Let $\C$ be the class of all cardinals $\alpha \geq \lambda_0$ such that $\alpha = \alpha^*$. As in the proof of \cite[Theorem 1]{AR1}, we easily check that the class $\C$ is closed and unbounded. Since the class of all regular cardinals is stationary under Vop\v enka's principle, it follows that $\C$ contains a regular cardinal $\lambda$ (see \cite[Theorem 3]{Ma}). This regular cardinal $\lambda$ satisfies (3). Since $\lambda_0$ can be chosen to be arbitrarily large, the same holds for $\lambda$, too. 
\end{proof} 
 
\begin{remark}\label{weakly}
In Theorem \ref{swfs-general case}(2), the objects of $\ca(\cx, \lambda)$ are $\lambda$-presentable in $\cm$ by construction. 
We cannot claim that they are $\lambda$-presentable in $\cm_c$ because the inclusion $\cm_c\to\cm$ need not 
preserve $\lambda$-directed colimits. 
\end{remark}

The following statement is well known in the context of combinatorial model categories (cf. \cite[Proposition 7.2]{D1}) and 
will be useful in the proofs of our main results. We state it here more generally for weak factorization systems.  

\begin{corollary}\label{swfs-factorizations}
Let $\cm$ be a cocomplete category and $(\cl, \crr)$ a strictly cofibrantly generated weak factorization system. 
\begin{itemize}
\item[(1)] Assuming Vop\v{e}nka's principle, for arbitrarily large regular cardinals $\lambda$, the following holds: every morphism $f: X \to Y$ between 
$\lambda$-presentable objects in $\cm_c$ admits a functorial factorization $X \stackrel{i}{\to} Z \stackrel{p}{\to}
Y$ with $i \in \cl$, $p \in \crr$ and $Z$ is again $\lambda$-presentable in $\cm_c$. 
\item[(2)] Suppose that $(\cl, \crr)$ is perfectly cofibrantly generated. Then for arbitrarily large regular cardinals $\lambda$, the following holds:
every morphism $f: X \to Y$ between cofibrant objects which are $\lambda$-presentable in $\cm$ admits a functorial factorization $X \stackrel{i}{\to} Z \stackrel{p}{\to}
Y$ with $i \in \cl$, $p \in \crr$ and $Z \in \cm_c$ is again $\lambda$-presentable in $\cm$.
\end{itemize}
\end{corollary}
\begin{proof}
We first prove (2) whose proof is based on the small object argument and is similar to the combinatorial case. Let $\cx \subseteq \cl$ be a generator for $\crr$ such that 
the domains and codomains of $\cx$ are $\mu$-presentable. For every morphism $f:X\to Y$ in $\cm$ form the colimit $F_0f$ 
of the diagram
$$
\xymatrix{
X \\
&&\\
A\ar[uu]^u  \ar@{.}[ur] \ar [rr]_h && B
}
$$
consisting of all spans $(u,h)$ with $h:A\to B$ in $\cx$, one for each commutative square 
$$
\xymatrix{
A \ar[d]_h \ar[r]^u & X \ar[d]^f \\
B \ar[r]^v & Y
}
$$
Let $\alpha_{0f}:X\to F_0f$ denote the component of the colimit cocone 
(the other components are morphisms $B\to F_0f$ and the squares
$$
\xymatrix@C=3pc@R=3pc{
X \ar[r]^{\alpha_{0f}} & F_0f \\
A \ar [u]^{u} \ar [r]_{h} &
B \ar[u]_{}
}
$$
are commutative). Let $\beta_{0f}:F_0f\to Y$ be the morphism induced by the morphisms $f$ and $v$'s. This way we get a functor $F_0:\cm^{\to}\to\cm$
which preserves $\mu$-filtered colimits. Let $F_if,\alpha_{if}$ and $\beta_{if}$, $i\leq\mu$, be given by the following transfinite induction:
$$F_{i+1}f =F_0\beta_{if}, \ \  \alpha_{i+1,f}=\alpha_{0, \beta_{if}}\alpha_{if}, \ \ \beta_{i+1,f}=\beta_{0,\beta_{if}}$$
and the limit step is given by taking colimits. Then the functor $F_\mu:\cm^{\to}\to\cm$ preserves $\mu$-filtered colimits.
For each $f: X \to Y$, the resulting functorial factorization $$X \stackrel{i}{\to} F_{\mu}(f) \stackrel{p}{\to} Y$$ has $i \in \cl$ and $p \in \crr$ by the small object argument. 
The existence of $\lambda$ can be argued similarly to \cite[Theorem 2.19]{AR} whose proof is also valid in this context (see Proposition \ref{uniformization2} below for a suitable 
formulation). Indeed, it suffices to check that \cite[Remark 2.15]{AR} can be used. In other words, it suffices to show that every morphism between cofibrant objects is a $\mu$-directed colimit of morphisms between cofibrant objects which are $\mu$-presentable in 
$\cm$. This is guaranteed by Theorem \ref{swfs-general case}(2) applied to the weak factorization system on $\cm^{\to}$ where the cofibrations are defined pointwise - this is again perfectly $\mu$-cofibrantly generated.

The proof of (1) starts similarly. Let $\cx \subseteq \cl$ be a generator for $\crr$ such that the domains  of $\cx$ are strictly
$\mu$-small relative to $\cx$. For every morphism $f:X\to Y$ in $\cm_c$, we can proceed as before and define a functor $F_\mu:\cm_c^{\to}\to\cm_c$. But we will need to use a different approach in order to obtain $\lambda$. 

By Theorem \ref{swfs-general case}(3), the category $\cm_c$ is preaccessible. As in the proof of Theorem \ref{swfs-general case}(3), 
we may assume that $\cm_c$ is a full subcategory of the category of graphs $\Gra$ such that the inclusion $\cm_c \to \Gra$ preserves $\lambda_0$-directed colimits for some $\lambda_0$. For $X\in\cm_c$, $|X|$ denotes the cardinality of the underlying set of the graph $X$. Given a regular cardinal $\alpha \geq \lambda_0$,  $|X|<\alpha$ implies that $X$ is $\alpha$-presentable in $\cm_c$. The proof of \cite[Theorem 1]{AR1} considered a closed and unbounded class of cardinals $\C_1$ with the property that for each $\alpha \in \C_1$, every subobject $Y \subseteq X$ in $\Gra$ with $X \in \cm_c$ and $|Y| < \alpha$ can be embedded into a strong subobject $Y^* \subseteq X$ in $\cm_c$ with $|Y^*| < \alpha$. Assuming Vop\v{e}nka's principle, it follows that there is a regular cardinal $\lambda$ in $\C_1$ (see the proof of \cite[Theorem 1]{AR1}). 
Given such a regular cardinal $\lambda$, each object $X \in \cm_c$ is a $\lambda$-directed colimit of objects in $\cm_c$ with cardinality (as graphs) $< \lambda$. Thus, if $X$ is $\lambda$-presentable object in $\cm_c$, then it is a split subobject of an object with cardinality $< \lambda$, and therefore also $|X| < \lambda$. 

For \emph{any} cardinal $\alpha$, let $\alpha^\ast \geq \alpha$ be the smallest cardinal such that $|X|,|Y|<\alpha$ implies that $|F_{\mu}(f)|<\alpha^\ast$ for every $f: X \to Y$, where $X\to F_{\mu}(f) \to Y$ is the factorization constructed above. 
Let $\C_2$ be the class of all cardinals $\alpha \geq \lambda_0$ such that $\alpha=\alpha^\ast$. As in the proof of \cite[Theorem 1]{AR1}, we easily check that the class $\C_2$ is closed and unbounded. The intersection $\C_3 = \C_1 \cap \C_2$ is again closed and unbounded. Since the class of all regular cardinals is stationary under Vop\v enka's principle (see \cite[Theorem 3]{Ma}), it follows that $\C_3$ contains a regular cardinal $\lambda$, and this has the required property. Since $\lambda_0$ can be chosen to be arbitrarily large, the same applies to $\lambda$, too. 

\end{proof}

\begin{corollary}\label{swfs-special case} 
Let $(\cm, \mathcal{C}of, \mathcal{W}, \mathcal{F}ib)$ be a model category. Suppose that the weak factorization system $(\mathcal{C}of, \cw \cap \mathcal{F}ib)$ 
is strictly $\kappa$-cofibrantly generated with generating set $\cx$. Then:
\begin{itemize}
 \item[(1)] For every regular cardinal $\lambda \geq \kappa$, every object of $\cm_c$ is a 
$\lambda$-filtered colimit of objects in $\ca(\cx, \lambda)$. 
 \item[(2)] Suppose that $(\mathcal{C}of, \cw \cap \mathcal{F}ib)$ is perfectly $\kappa$-cofibrantly generated with generating set $\cx$ and 
 let $\lambda \geq \kappa$ be a regular cardinal. Then the small colimit-dense subcategory $\ca(\cx, \lambda) \subseteq \cm_c$ consists of $\lambda$-presentable objects in $\cm$. 
 In particular, every cofibrant object is presentable in $\cm$. 
 \item[(3)] Assuming Vop\v enka's principle, $\cm_c$ is preaccessible and therefore every cofibrant object is presentable in $\cm_c$. Moreover, there are arbitrarily large 
regular cardinals $\lambda$ such that every object of $\ca(\cx, \lambda)$ is $\lambda$-presentable in $\cm_c$. 
\item[(4)] Suppose that every object is cofibrant and $(\mathcal{C}of \cap \cw, \mathcal{F}ib)$ is setwise cofibrantly generated. 
Then, assuming Vop\v enka's principle, $\cm$ is a combinatorial model category.
 \end{itemize}
\end{corollary}
\begin{proof}
(1)-(3) are special cases of Theorem \ref{swfs-general case}. For (4),  Theorem \ref{swfs-general case} yields that $\cm=\cm_c$ is 
locally presentable and every object is presentable. In particular, the weak 
factorization system $(\mathcal{C}of \cap \cw, \mathcal{F}ib)$ is cofibrantly generated.
\end{proof}

\begin{remark} \label{same-cardinal}

By combining the arguments in Corollaries \ref{swfs-factorizations}(1) and \ref{swfs-special case}(3) (or by 
applying Proposition \ref{uniformization} below), it follows similarly that there are arbitrarily large regular 
cardinals $\lambda$ for which both Corollary \ref{swfs-factorizations}(1) for $(\cl, \crr) = (\mathcal{C}of, \cw \cap \mathcal{F}ib)$ and Corollary \ref{swfs-special case}(3) are satisfied.
\end{remark}

\subsection{Functors between preaccessible categories} 

\medskip 

We digress slightly to add some comments on the arguments 
that were used in the  proofs of the previous subsection.  The comments concern certain variations of the Uniformization Theorem for accessible categories \cite[Theorem 2.19]{AR} and may be of independent interest too. 

First we introduce some terminology which is inspired from the context of Theorem \ref{swfs-general case}. Given a category $\cm$ and a full subcategory $\cb \subseteq \cm$, we denote by $\mathbf{Pres}^{\cm}_{\lambda}(\cb)$ the full subcategory of $\cb$ spanned by those objects which are $\lambda$-presentable in $\cm$. If $\cm$ admits $\kappa$-filtered colimits, we say that $\cb$ is $\kappa$-\emph{preaccessible in} $\cm$ if there is a small full subcategory $\ca \subseteq \cb$ which consists of $\kappa$-presentable objects in $\cm$ such that every object in $\cb$ is a $\kappa$-filtered colimit in $\cm$ of objects that lie in $\ca$. 

The following result can be regarded as a kind of  ``relative" Uniformization Theorem.

\begin{proposition} \label{uniformization2}
Let $\cm$ and $\cn$ be categories which admit $\kappa$-filtered colimits for some regular cardinal $\kappa$, and $$F_i \colon \cm \to \cn, \ \ i \in I,$$ a small collection of functors which preserve $\kappa$-filtered colimits. Let $\cb \subseteq \cm$ be a full subcategory which is $\kappa$-preaccessibe in $\cm$ and suppose that for every $X \in \cb$, $F_i(X) \in \cn$
is presentable in $\cn$ for every $i \in I$. Then there are arbitrarily large regular cardinals $\lambda$ such that for every $X \in \mathbf{Pres}^{\cm}_{\lambda} (\cb)$, $F_i(X)$ is $\lambda$-presentable in $\cn$ 
for every $i \in I$.
\end{proposition}
\begin{proof}
The proof follows \cite[Theorem 2.19, Remark 2.15]{AR}. Let $\ca \subseteq \cb$ be a small full subcategory of $\kappa$-presentable objects in $\cm$ such that every object in $\cb$ is a $\kappa$-directed colimit of objects from $\ca$. Let $\kappa_1 \triangleright \kappa$ be a regular cardinal such that for each $X \in \ca$, $F_i(X) \in \cn$ is $\kappa_1$-presentable in $\cn$ for all $i \in I$. We claim that each $\lambda \triangleright \kappa_1$ has the required property. We 
recall that for each regular cardinal $\mu$, we have $\mu \triangleleft \mu^+$ where $\mu^+$ denotes the cardinal successor of $\mu$. We refer to \cite[2.11-2.13]{AR} for the definition and properties of the relation $\triangleright$. 

Let $X \in \cb$ be $\lambda$-presentable in $\cm$. By assumption, $X$ is a $\kappa$-directed 
colimit of a diagram $D: P \to \cm$ whose values are in $\ca$. We consider the new poset $\hat{P}$ of all $\kappa$-directed $\lambda$-small subsets of $P$ ordered by inclusion. 
Then $\hat{P}$ is $\lambda$-directed. We define a new diagram $\hat{D}: \hat{P} \to \cm$ which sends a subset $Q \subseteq P$ to the colimit in $\cm$ 
of the restriction of $D$ to $Q$. This defines a $\lambda$-directed diagram $\hat{D}$ whose colimit is again $X$. As a consequence, the identity of $X$ factors through some stage of this 
diagram $\hat{D}$ and therefore $X$ is a split subobject of a $\lambda$-small $\kappa$-directed colimit of objects in $\ca$. For each $i \in I$, $F_i(X)$ is then a 
split subobject of a $\lambda$-small $\kappa$-directed colimit of $\kappa_1$-presentable objects in $\cn$. Therefore $F_i(X)$ is $\lambda$-presentable for every $i \in I$, as required. 
\end{proof}

Secondly, the argument that was used in the proofs of Theorem \ref{swfs-general case}(3) and Corollary \ref{swfs-factorizations}(1) generalizes to obtain the following more general 
result about functors between preaccessible categories under the assumption of Vop\v{e}nka's principle. It is an analogue in
this context of the Uniformization Theorem for accessible functors (see \cite[Theorem 2.19]{AR}).

\begin{proposition} \label{uniformization}
Let $F_i: \cm_i \to \cn$, $i \in I$, be a small collection of functors between preaccessible categories $\cm_i$, $i \in I$, and $\cn$. 
Assume that Vop\v{e}nka's principle holds. Then there are arbitrarily large regular cardinals $\lambda$ such that 
each functor $F_i$ preserves $\lambda$-presentable objects.
\end{proposition}
\begin{proof}
As in the proof of Theorem \ref{swfs-general case}(3), we may assume that $\cm_i$, $i \in I$, and $\cn$ are full subcategories of the category of graphs $\Gra$ since each of them has a small dense subcategory. Assuming Vop\v{e}nka's principle, the inclusions $\cm_i \to \Gra$, $i \in I$, and $\cn \to \Gra$ preserve $\lambda_0$-directed colimits for some $\lambda_0$, by \cite[Theorem 6.9]{AR}. 
For $X \in \Gra$, $|X|$ denotes the cardinality of the underlying set of the graph $X$. As before, given a regular
cardinal $\alpha \geq \lambda_0$ and $X \in \cm_i$ ($X \in \cn$),  $|X|<\alpha$ implies that $X$ is $\alpha$-presentable in $\cm_i$ (in $\cn$). Moreover, as already explained in the proof of Corollary \ref{swfs-factorizations}(1), the proof of \cite[Theorem 1]{AR1} shows how to define a closed and unbounded class of cardinals $\C_1$ with the property that for each regular cardinal $\lambda \in \C_1$, we have:
(a) for any $i \in I$, $X \in \cm_i$ is $\lambda$-presentable if and only if $|X| < \lambda$, and (b) $Y \in \cn$ is $\lambda$-presentable 
if and only if $|Y| < \lambda$. 

For \emph{any} cardinal $\alpha$, let $\alpha^\ast \geq \alpha$ be the smallest cardinal such that for 
each $X \in \cm_i$ with $|X| <\alpha$, then $|F_i (X)|<\alpha^\ast$ for every $i \in I$. Let $\C_2$ be the class of all cardinals $\alpha \geq \lambda_0$ such that $\alpha=\alpha^\ast$. The class $\C_2$ is closed and unbounded (similarly to the proof in 
\cite[Theorem 1]{AR1}). The intersection $\C_3 = \C_1 \cap \C_2$ is again closed and unbounded. Since the class of all regular cardinals 
is stationary under Vop\v enka's principle, it follows that 
$\C_3$ contains a regular cardinal $\lambda$ (see \cite{Ma}, \cite{AR1}). This $\lambda$ has the required 
property. Since $\lambda_0$ can be chosen to be arbitrarily large, the same applies to $\lambda$, too. 

\end{proof}

\begin{example}
Let $\mathscr{P} \colon \Set \to \Set$ be the covariant power set functor. For an uncountable regular cardinal $\lambda$, 
the functor $\mathscr{P}$ preserves $\lambda$-presentable objects in $\Set$ if and only if $\lambda$ is inaccessible. Applying 
Proposition \ref{uniformization} to this functor, we can then deduce the well known fact that Vop\v{e}nka's principle implies 
the existence of arbitrarily large inaccessible cardinals (see \cite[Appendix]{AR}). In fact, a careful examination of the arguments 
shows that the existence of arbitrarily large inaccessible cardinals 
follows also from the (strictly weaker) assertion that the class of all regular cardinals is stationary. The fact that this assertion is a consequence of Vop\v{e}nka's principle is easily deduced from \cite[Theorem 3]{Ma}.
\end{example}

\subsection{Weak equivalences and filtered colimits} 

\medskip

Our next goal is to find general conditions under which the class of weak equivalences in a model category $\cm$ is closed under large enough filtered colimits. 
The strategy follows the proof of \cite[Proposition 4.1]{RR}. We begin with the following lemma about filtered colimits of trivial fibrations. 

\begin{lemma}\label{triv-fib}
Let $(\cm, \mathcal{C}of, \mathcal{W}, \mathcal{F}ib)$ be a model category. Suppose that $(\mathcal{C}of, \cw \cap \mathcal{F}ib)$ is strictly cofibrantly generated. Then, assuming Vo\-p\v en\-ka's principle, there is a regular 
cardinal $\lambda$ such that given a $\lambda$-filtered diagram $F: P \to \cm_c^{\to}$ whose values are trivial fibrations in $\cm$ and $\colim_P F \in \cm_c^{\to}$, then $\colim_P F$  is again a trivial fibration.
\end{lemma}
\begin{proof}
Let $\cx$ be a generating set  for $(\mathcal{C}of, \cw \cap \mathcal{F}ib)$ whose domains are strictly $\kappa$-small and let $\ca = \ca(\cx, \kappa)$ be the associated set of objects. 
By Theorem \ref{swfs-general case}, $\ca$ is a small colimit-dense subcategory of $\cm_c$.
Let $\ck$ be the full subcategory of $\cm$ obtained by adding the domains and the codomains of the morphisms in $\cx$ to $\cm_c$. The union of the 
domains and codomains of $\cx$ with $\ca$ defines a small colimit-dense subcategory of $\ck$. Using the same argument as in the proof of Theorem \ref{swfs-general case}(3), it follows 
that every object of $\ck$ is presentable in $\ck$. Then there exists a regular cardinal $\lambda$ such that the domains and the codomains of the morphisms in $\cx$ are 
$\lambda$-presentable in $\ck$. Consequently, trivial fibrations between cofibrant objects are closed under $\lambda$-filtered colimits which exist in $\ck$. In particular, they are 
closed under those $\lambda$-filtered colimits that come from colimits in $\cm$. 
\end{proof}

\begin{remark} \label{triv-fib2}
Note that an easier and direct argument shows that if $(\mathcal{C}of, \cw \cap \mathcal{F}ib)$ is perfectly cofibrantly generated, then the conclusion of Lemma \ref{triv-fib} holds 
for all trivial fibrations without the assumption of Vo\-p\v en\-ka's principle.
\end{remark}

\begin{remark}\label{fib-case}
There is an analogous statements for fibrations assuming instead that $(\mathcal{C}of \cap \cw, \mathcal{F}ib)$ is strictly cofibrantly generated. 
The proof is similar to the proof of Lemma \ref{triv-fib}. 
\end{remark}

\begin{proposition}\label{weak-eq} 
Let $(\cm, \mathcal{C}of, \mathcal{W}, \mathcal{F}ib)$ be a model category. Suppose that $(\mathcal{C}of, \cw \cap \mathcal{F}ib)$ is strictly cofibrantly generated. Then, 
assuming Vo\-p\v en\-ka's principle, there is a regular cardinal $\lambda$ such that given a $\lambda$-filtered diagram 
$F: P \to \cm_c^{\to}$ whose values are weak equivalences and $\colim_P F \in \cm_c^{\to}$, then $\colim_P F$ is again 
a weak equivalence. 
\end{proposition}
\begin{proof}
Let $\lambda$ be a regular cardinal that satisfies the property stated in Lemma \ref{triv-fib}. It suffices to prove the claim
for $\lambda$-directed colimits. Following Remark \ref{cofinal}, for every $\lambda$-directed category $\cd$, there is a cofinal 
functor $\cd_0 \to \cd$ from a $\lambda$-directed good poset $\cd_0$, so it suffices 
to prove the claim for $\lambda$-directed good colimits.  

Let $P$ be a $\lambda$-directed good poset. The category $\cm^P$ of $P$-diagrams in
$\cm$ admits a model structure where the weak equivalences and the fibrations are defined pointwise (see \cite[Theorem 5.1.3]{Ho}). Furthermore, every cofibration in $\cm^P$ is also a pointwise cofibration. Then the colimit-functor 
$$\mathrm{colim}_P: \cm^P \to \cm$$
is a left Quillen functor.  

Let $F, G: P \to \cm_c$ be two pointwise cofibrant $P$-diagrams whose colimits 
are in $\cm_c$ and $\phi: F \to G$ a natural transformation which is a pointwise weak equivalence. There is a factorization of $\phi$ in $\cm^P$
$$F \stackrel{\iota}{\longrightarrow} T \stackrel{\pi}{\longrightarrow} G$$ 
such that $\iota$ is a trivial cofibration and $\pi$ is a trivial fibration. The diagram $T: P \to \cm$ is also pointwise cofibrant. Since $\mathrm{colim}_P$ is a 
left Quillen functor, the induced morphism $$\mathrm{colim}_P(F) \to \mathrm{colim}_P (T)$$
is a trivial cofibration. Therefore $\mathrm{colim}_P(T)$ is also cofibrant. By Lemma \ref{triv-fib}, the induced morphism 
$$\mathrm{colim}_P (T) \to \mathrm{colim}_P (G)$$
is a trivial fibration (between cofibrant objects). Hence $\mathrm{colim}_P(\phi)$ is a weak equivalence as required.
\end{proof}

\begin{remark} \label{weak-eq2}
Following Remark \ref{triv-fib2}, if $(\mathcal{C}of, \cw \cap \mathcal{F}ib)$ is perfectly cofibrantly generated, then the conclusion of Proposition 
\ref{weak-eq} holds for all weak equivalences without the assumption of Vop\v enka's principle.
\end{remark}

\begin{corollary} \label{homotopy-colim}
Let $(\cm, \mathcal{C}of, \mathcal{W}, \mathcal{F}ib)$ and $\lambda$ be as in Proposition \ref{weak-eq}. Let $F: P \to \cm_c$ be a $\lambda$-filtered diagram such that $\colim_P F$ is cofibrant. Then the canonical morphism 
$$c: \hocolim_P F \longrightarrow \colim_P F$$
is a weak equivalence.
\end{corollary}
\begin{proof}
Similarly to the proof of Proposition \ref{weak-eq}, that is, by (homotopy) cofinality, we may assume that $P$ is a $\lambda$-directed good poset. Let 
$\widetilde{F}: P \to \cm_c$ be a cofibrant 
replacement of $F$ in the model category $\cm^P$ where weak equivalences and fibrations are defined pointwise, and let $\eta \colon \widetilde{F} \to F$ 
denote the weak equivalence in $\cm^P$. Then it suffices to show that the canonical morphism 
$$\colim_P(\eta) \colon \colim_P \widetilde{F} \to \colim_P F$$
is a weak equivalence. Note that $\colim_P \widetilde{F}$ is cofibrant. Then the result follows from Proposition \ref{weak-eq}.
\end{proof}

\section{Small Presentations of Model Categories} \label{section4}

A \textit{small presentation} of a model category $\cm$ consists of a small category $\cc$, a set of morphisms $S$ in the projective model category $\mathcal{SS}et^{\cc^{\op}}$, and a Quillen equivalence 
$$F: L_S \mathcal{SS}et^{\cc^{\op}} \rightleftarrows \cm :G.$$
Here $\mathcal{SS}et$ denotes the usual combinatorial model category of simplicial sets and $L_S \SSet^{\cc^{\op}}$ the left Bousfield localization of the projective model category $\SSet^{\cc^{\op}}$ at the set $S$. It is 
well known that $\SSet^{\cc^{\op}}$ is a (proper, simplicial) combinatorial model category. Moreover, when $S$ is a set of morphisms, 
the left Bousfield localization $L_S \SSet^{\cc^{\op}}$ exists and defines a new (left proper, simplicial) combinatorial model category.
The notion of a small presentation for model categories was introduced and studied  by Dugger \cite{D, D1}. It is
the model-categorical analogue of the notion of a presentation of a locally presentable category as a small orthogonality class in a presheaf category. 

In this section, we obtain results about the existence of small presentations for model categories - assuming Vop\v{e}nka's principle or not. 
As a consequence, we conclude that suitable model categories are Quillen equivalent to a (left proper, simplicial) combinatorial model category. 
The strategy for obtaining such small presentations follows Dugger \cite{D1} who showed that every combinatorial model category admits a small
presentation.

\subsection{Dugger's method} \label{method}

\medskip

Let $\cc$ be a small category and $\SSet^{\cc^{\op}}$ the projective model category of simplicial presheaves on $\cc$ where the weak equivalences and the fibrations are defined pointwise. 
Given a cocomplete category $\cm$, a left adjoint functor 
$F: \SSet^{\cc^{\op}} \to \cm$ is determined, uniquely up to canonical isomorphism, by its restriction $$f = F \circ y: \cc \times \Delta \to \cm$$ where 
$y: \cc \times \Delta \to \SSet^{\cc^{\op}}=\Set^{(\cc \times \Delta)^{\op}}$ is the Yoneda embedding. 

Let $\cm$ be a model category and $f: \cc \times \Delta \to \cm$ a diagram. The associated left adjoint functor $F: \SSet^{\cc^{\op}} \to \cm$ is left Quillen if and only if for each 
$c \in \cc$, the cosimplicial object $f(c, [\bullet]) \in c \cm : = \cm^{\Delta}$ is a cosimplicial resolution in $c \cm$, i.e., it is Reedy cofibrant and homotopically constant, see \cite[Proposition 3.4]{D}. Such a cosimplicial resolution is determined up to 
a contractible space of choices by its restriction to degree $0$, i.e., the diagram $f(-, [0]): \cc \to \cm, \ c \mapsto f(c, [0])$. 

These observations suggest the following general way of constructing Quillen adjunctions $\SSet^{\cc^{\op}} \rightleftarrows \cm$ (see also \cite[Section 6]{D1}). Let $C$ be a small full 
subcategory of $\cm_c$ and let $\Delta C$ denote the small full subcategory of $c \cm$ which consists of cosimplicial resolutions $A^{\bullet} \in c \cm$ such that $A^n \in C$ for all $n$. 
Then the functor 
$f: \Delta C \times \Delta \to \cm$, $(A^{\bullet}, [n]) \mapsto A^n$, extends to a Quillen adjunction
$$F_C: \SSet^{\Delta C^{\op}} \rightleftarrows \cm: G_C.$$
We will refer to this Quillen adjunction as the \emph{canonical Quillen adjunction which is generated by $C \subset \cm_c$}. 

\medskip 

The first step towards finding a small presentation for $\cm$ is to find a left Quillen functor $F: \SSet^{\cc^{\op}} \to 
\cm$ which defines a homotopy reflection in the following sense.

\begin{definition}
A Quillen adjunction $F: \SSet^{\cc^{\op}} \rightleftarrows \cm: G$ is a \textit{homotopy reflection} if for every fibrant object 
$X \in \cm$ and cofibrant replacement $G(X)^{c} \stackrel{\sim}{\to} G(X)$, the induced morphism $F(G(X)^{c}) 
\to X$ is a weak equivalence. That is, the induced derived adjunction $\mathbb{L}F: \Ho(\SSet^{\cc^{\op}}) \rightleftarrows 
\Ho(\cm): \mathbb{R}G$ defines a reflection. 
\end{definition}

Such a left Quillen functor $F$ is called \emph{homotopically surjective} in \cite{D1}. 
Following \cite{D1}, one can identify the derived counit transformation of a Quillen adjunction $F: \SSet^{\cc^{\op}} \rightleftarrows \cm: G$ 
as follows. Let $f = F \circ y : \cc \times \Delta \to \cm$ be the restriction along the Yoneda embedding. For each $X \in \cm$, 
we have the canonical diagram 
$$\cc \times \Delta \downarrow X \longrightarrow \cm, \ (c, [n]) \mapsto f(c, [n]).$$
Let $\mathrm{hocolim}(\cc \times \Delta \downarrow X)$ denote its homotopy colimit. 
In \cite[4.2-4.4]{D1}, the derived counit transformation of $(F, G)$ is identified with the canonical natural morphism 
$$q_X: \mathrm{hocolim}(\cc \times \Delta \downarrow X) \to X.$$
Hence a Quillen adjunction $F: \SSet^{\cc^{\op}} \rightleftarrows \cm: G$ is a homotopy reflection if and only if the canonical morphism 
$q_X$ is a weak equivalence for all fibrant $X$. 

\medskip

Using this identification and combining several results from \cite{D1}, we obtain the following criterion for a canonical Quillen 
adjunction $(F_C, G_C)$ to be a homotopy refection. We recall that the category $c \cm$ has a simplicial enrichment and it is tensored 
and cotensored over simplicial sets. For every cosimplicial resolution $A^{\bullet} \in c\cm$ and simplicial 
set $K$, the cosimplicial object $A^{\bullet} \otimes K$ is again a cosimplicial resolution (see \cite[Proposition 4.8]{D2}). 

\begin{proposition} \label{auxiliary-comparison}
Let $\cm$ be a model category, $C$ be a small full subcategory of $\cm_c$ and $(F_C, G_C): \SSet^{\Delta C^{\op}} \rightleftarrows \cm$ the associated
canonical Quillen adjunction. Suppose that the following are satisfied:
\begin{itemize}  
 \item[(a)] there is a functorial Reedy cofibrant replacement for cosimplicial objects in $C$, that is, a functor $R \colon c \cc 
\to c \cc$ such that $R(X^{\bullet})$ is Reedy cofibrant and a natural weak equivalence $\eta \colon R(X^{\bullet}) \to X^{\bullet}$. Moreover, we require that $R(X^{\bullet})^0 = X^0$ and $\eta$ is the identity in degree $0$. 
 \item[(b)] for each $A^{\bullet} \in \Delta C$, then the cosimplicial resolution $A^{\bullet} \otimes \Delta^1$ is again in $\Delta C$.
\end{itemize}
Then $(F_C, G_C)$ is a homotopy reflection if and only if the canonical morphism 
$$\hocolim(C \downarrow X) \to X$$
is a weak equivalence for each fibrant-cofibrant object $X$.
\end{proposition}
\begin{proof}
Let $\Delta C^n \subset \Delta C \times \Delta$ be the full subcategory of objects $(A^{\bullet}, [n])$. Clearly, $\Delta C^0 = \Delta C$. Note that there are canonical 
functors $\Delta C^n \to \cm$, $(A^{\bullet}, [n]) \mapsto A^n$. For each $X \in \cm$, there is a natural functor $\Delta C \downarrow X \to \Delta C^n 
\downarrow X$ which is given by $(A^{\bullet}, A^0 \to X) \mapsto (A^{\bullet}, A^n \to A^0 \to X)$. Using property (b), 
the proof of \cite[Lemma 6.4]{D1} shows that the canonical induced morphism
$$\hocolim(\Delta C \downarrow X) \longrightarrow \hocolim(\Delta C^n \downarrow X)$$
is a weak equivalence. Applying \cite[Proposition 4.6]{D1}, it follows that there is a canonical weak equivalence 
$$\hocolim(\Delta C \downarrow X) \simeq \hocolim(\Delta C \times \Delta \downarrow X).$$
Therefore $(F_C, G_C)$ is a homotopy reflection if and only if the canonical morphism 
$$\mathrm{hocolim}(\Delta C \downarrow X) \longrightarrow X$$
is a weak equivalence for each fibrant $X \in \cm$. This follows from the identification of the derived counit morphism (\cite[Proposition 4.2, Corollary 4.4]{D1}). Furthermore, using property (a), the proof of \cite[Lemma 6.3]{D1} shows that the canonical morphism induced by $\Delta C \to C$, 
$A^{\bullet} \mapsto A^0$, 
$$\hocolim(\Delta C \downarrow X) \longrightarrow \hocolim(C \downarrow X)$$
is a weak equivalence. The restriction to cofibrant objects does not affect the argument by \cite[Corollary 4.3]{D1}.
\end{proof}

\begin{remark} \label{explaining-argument}
It is useful to note that the assumptions (a) and (b) in Proposition \ref{auxiliary-comparison} are used separately in 
the proof for two different arguments. In particular, if only (b) is satisfied, then $(F_C, G_C)$ is a homotopy reflection 
if and only if the canonical morphism
$$\hocolim(\Delta C \downarrow X) \to X$$
is a weak equivalence for each fibrant-cofibrant $X$. Moreover, in the proof of Theorem \ref{main}(2) below, a weaker version 
of (a) will be sufficient in that context.
\end{remark}

Thus, in order to find a canonical Quillen adjunction which is a homotopy reflection, the strategy will first be to find a small full subcategory $C$ of $\cm$ 
which satisfies the assumptions of Proposition \ref{auxiliary-comparison}. When this is done and we obtain a homotopy reflection, what remains is to address the problem of turning a homotopy reflection into a Quillen equivalence. The following generalizes  \cite[Proposition 3.2]{D1}.

\begin{proposition} \label{localization}
Let $(\cm, \mathcal{C}of, \mathcal{W}, \mathcal{F}ib)$ be a model category, $\cc$ be a small category, and $$F: \SSet^{\cc^{\op}} \rightleftarrows \cm: G$$ 
a homotopy reflection. Then: 
\begin{itemize}
\item[(1)] Assuming Vop\v{e}nka's principle, there is a set of morphisms $S \subset \SSet^{\cc^{\op}}$ and an induced 
Quillen equivalence $L_S \SSet^{\cc^{\op}} \rightleftarrows \cm$. 
\item[(2)] The same conclusion holds, without the assumption of Vop\v{e}nka's principle, if in addition $(\mathcal{C}of, \cw \cap \mathcal{F}ib)$ is perfectly cofibrantly generated 
and $(\mathcal{C}of \cap \cw, \mathcal{F}ib)$ is semi-perfectly cofibrantly generated. 
\end{itemize}
\end{proposition}
\begin{proof}
For (1), let $S$ be the class of canonical morphisms 
$$
X^c\to G\bigl( F(X^{c})^{f}\bigr)
$$ 
where $X$ is in $\SSet^{\cc^{\op}}$ and $(-)^{c}$ and $(-)^f$ denote choices of cofibrant and fibrant replacement functors. By \cite[Theorem 2.3]{RT} (see also \cite{CC}), the left Bousfield localization of 
$\SSet^{\cc^{\op}}$ at $S$ exists under the assumption of Vop\v{e}nka´s principle. Moreover, there is a set $S'\subset S$ 
such that we have $L_S \SSet^{\cc^{\op}} = L_{S'} \SSet^{\cc^{\op}}$. Then it is easy to conclude that the induced Quillen adjunction 
$L_S \SSet^{\cc^{\op}} \rightleftarrows \cm$ is a Quillen equivalence. 

For (2),  note first that there is a regular cardinal $\lambda$ such that the following are satisfied: 
\begin{enumerate}
 \item[(i)] there is a cofibrant replacement functor $(-)^c: \SSet^{\cc^{\op}} \to \SSet^{\cc^{\op}}$ which preserves $\lambda$-filtered colimits. 
 \item[(ii)] there is a fibrant replacement functor $(-)^f: \cm \to \cm$ which preserves $\lambda$-filtered colimits.
 \item[(iii)] $(F\circ y)(c, [n]) \in \cm_c$ is $\lambda$-presentable in $\cm$ for each $(c, [n]) \in \cc \times \Delta$.
\end{enumerate}
Indeed, conditions (i)-(ii) are satisfied by the factorizations given by the small object argument, and condition (iii) follows from 
Corollary \ref{swfs-special case}(2). Let $S$ be the set of morphisms 
$$
X^c\to G(F(X^c)^{f})
$$ 
where $X$ is $\lambda$-presentable in $\SSet^{\cc^{\op}}$. The left Bousfield localization $L_S \SSet^{\cc^{\op}}$ exists and there is an induced Quillen adjunction 
$\overline{F}: L_S \SSet^{\cc^{\op}} \rightleftarrows \cm: \overline{G}$ - this is again a homotopy reflection. We claim that the composite 
$$\SSet^{\cc^{\op}} \stackrel{(-)^c}{\longrightarrow} \SSet^{\cc^{\op}} \stackrel{F}{\longrightarrow} \cm \stackrel{(-)^f}{\longrightarrow} \cm \stackrel{G}{\longrightarrow} \SSet^{\cc^{\op}}$$
preserves $\lambda$-filtered colimits. Indeed, this follows from the assumptions since each of the functors preserves $\lambda$-filtered colimits. Therefore, if we write an arbitrary object
$X \in \SSet^{\cc^{\op}}$ as a $\lambda$-filtered colimit of $\lambda$-presentable objects, the canonical morphism $X^c \to G(F(X^c)^f)$ is a $\lambda$-filtered colimit of morphisms in $S$. Since weak equivalences in 
$\SSet^{\cc^{\op}}$ are closed under filtered colimits, it follows that $X^c \to G(F(X^c)^f)$ is an $S$-local equivalence for each $X$ and therefore $(\overline{F}, \overline{G})$ 
is a Quillen equivalence. 
\end{proof} 

\subsection{The main results} \label{results}

\medskip

The following theorem is our main result on the existence of small presentations for model categories. The first part assumes Vop\v{e}nka's principle and establishes a slightly weaker 
version of the claim in \cite{Ra}. The second part is independent of Vop\v{e}nka's principle and generalizes the main result of \cite{D1}.

\begin{theorem}\label{main}
Let $(\cm, \mathcal{C}of, \mathcal{W}, \mathcal{F}ib)$ be a model category such that $(\mathcal{C}of, \mathcal{W} \cap \mathcal{F}ib)$ is strictly cofibrantly generated.
\begin{itemize}
\item[(1)]  Assume that Vop\v enka's principle holds. Then $\cm$ admits a small presentation. 
\item[(2)] Suppose that $(\mathcal{C}of, \mathcal{W} \cap \mathcal{F}ib)$ is perfectly cofibrantly generated and $(\mathcal{C}of \cap \cw, \mathcal{F}ib)$ is semi-perfectly cofibrantly 
generated. Then $\cm$ admits a small presentation.
\end{itemize}
\end{theorem}
\begin{proof}
We first prove (1). By Corollary \ref{swfs-special case}(3) and Propositions \ref{uniformization} and \ref{weak-eq}, 
it follows that there is a regular cardinal $\lambda$ such that 
\begin{itemize} 
\item[(i)] $\cm_c$ contains a small full subcategory $\ca = \ca(\cx, \lambda)$ which consists of $\lambda$-presentable objects in $\cm_c$ and each cofibrant object is a $\lambda$-filtered colimit in $\cm$ of objects in $\ca$.
\item[(ii)] there is a Reedy cofibrant replacement functor $R \colon c \cm_c \to c \cm_c$ which preserves cosimplicial objects which are pointwise $\lambda$-presentable in $\cm_c$. Moreover, we may assume that $R(X^{\bullet})^0 \to X^0$ is the identity morphism for each $X^{\bullet} 
\in c \cm_c$ (since $X^0$ is already cofibrant in $\cm$). 
\item[(iii)] the weak equivalences between cofibrant objects are closed under $\lambda$-filtered colimits in $\cm$ which land in $\cm_c$.
\end{itemize}
(i) follows from Corollary \ref{swfs-special case}. Since $\cm_c$ is preaccessible, and therefore has a small dense subcategory, it embeds fully in the category of graphs $\Gra$. It follows that $c \cm_c$ embeds fully in a locally 
presentable category. As a consequence, assuming Vop\v{e}nka's principle, $c \cm_c$ is again preaccessible by \cite[Theorem 6.6]{AR} and \cite[Theorem 1]{AR1}. Let $R: c \cm_c \to c \cm_c$ be the standard Reedy cofibrant replacement functor which is 
defined inductively on the cosimplicial degree (leaving degree $0$ unchanged). Then (ii) follows from Proposition \ref{uniformization} applied to the small collection of functors 
$$\{ \mathrm{ev}_n R \colon c \cm_c \stackrel{R}{\to} c \cm_c \stackrel{\mathrm{ev}_n}{\to} \cm_c\}_{n \geq 0}.$$  Moreover, it is possible to find a common $\lambda$ satisfying simultaneously (i) and (ii) (cf. Remark \ref{same-cardinal}). (iii) follows from Proposition \ref{weak-eq}. 

Let $C$ be the small full subcategory of $\cm_c$ consisting of $\lambda$-presentable objects. We recall that $\Delta C$ denotes the small full subcategory of $c \cm$ which consists of cosimplicial 
resolutions $A^{\bullet} \in c\cm$ such that $A^n \in C$ for each $n$. As explained earlier, the functor $\Delta C \times \Delta \to \cm$, $(A^{\bullet}, [n]) \mapsto A^n$, extends to a 
left Quillen functor $F_C: \SSet^{\Delta C^{\op}} \to \cm$. We claim that this defines a homotopy reflection. Assumption (a) of Proposition \ref{auxiliary-comparison} is satisfied by (ii) above. Assumption (b) of Proposition \ref{auxiliary-comparison} is satisfied since $(A^{\bullet} \otimes \Delta^1)^n$ is the coend (in $\cm$) $A^{\bullet} \otimes_{\Delta} (\Delta^1 \times \Delta^n) \in \cm_c$, and therefore the presentability rank is preserved (this is also a coend in $\cm_c$). 

Therefore, by Proposition \ref{auxiliary-comparison}, it suffices to show that for each fibrant-cofibrant object $X \in \cm$, the canonical morphism
$$\mathrm{hocolim}(C \downarrow X) \longrightarrow X$$
is a weak equivalence. We have canonical morphisms
$$
\xymatrix{
\hocolim(C \downarrow X) \ar[dr] \ar[rr]^{\sim}  && \colim(C \downarrow X) \ar[dl]^{\cong}  \\
& X &
}
$$ 
The object $X$ can be written as a $\lambda$-filtered colimit of $G \colon I \to \cm$ with values in $\ca \subseteq C$ - this is also a colimit in $\cm_c$. It follows that the functor 
$$I \to C \downarrow X, \ \ i \mapsto (G(i) \to X)$$ 
is cofinal. Hence $C \downarrow X$ is also $\lambda$-filtered and the right morphism in the diagram above is an isomorphism by cofinality. Moreover, the horizontal morphism is a weak 
equivalence by (iii) above and Corollary \ref{homotopy-colim}. 
It follows that the left morphism is a weak equivalence for each fibrant-cofibrant $X$, and therefore $F$ defines a homotopy reflection as claimed. 
The result then follows by applying Proposition \ref{localization}(1).

For (2), the argument is analogous - but a little different. Again, there is a regular cardinal $\lambda$ such that 
\begin{itemize} 
\item[(i)] $\cm_c$ contains a small full subcategory $\ca = \ca(\cx, \lambda)$ which consists of $\lambda$-presentable objects in $\cm$ and each cofibrant object is a $\lambda$-filtered 
colimit in $\cm$ of objects in $\ca$.
\item[(ii)] there is a Reedy cofibrant replacement functor for constant cosimplicial objects at cofibrant objects, $R \colon \cm_c \to c \cm_c$, where $R(X) \to cX$ is a trivial Reedy fibration and $R$ sends $\lambda$-presentable cofibrant objects to cosimplicial objects which are pointwise $\lambda$-presentable in $\cm$. 
\item[(iii)] the weak equivalences are closed under $\lambda$-filtered colimits in $\cm$.
\end{itemize}
(i) and (iii) are guaranteed by Corollary \ref{swfs-special case}(2) and Remark \ref{weak-eq2} respectively. To see (ii), we need to consider the model category $c \cm$ with the Reedy model structure. The weak factorization system $(\mathcal{C}of_{c \cm}, \cw_{c\cm} \cap \mathcal{F}ib_{c \cm})$ of this Reedy model category $c \cm$ is again perfectly cofibrantly generated (see \cite[Proposition 15.6.24]{Hi}). Even though (ii) does not follow directly from Corollary 
\ref{swfs-factorizations}(2), the same idea of proof can be applied to this weak factorization system. Indeed, the functorial 
cofibrant replacement from the small object argument yields a functor $\widetilde{R}: \cm \to c \cm$ which restricts to a Reedy cofibrant replacement functor $R \colon \cm_c \to c \cm_c$ with the desired properties. 
The existence of the required $\lambda$ follows from Proposition \ref{uniformization2} using Corollary \ref{swfs-special case}(2). 

Let $C$ be the small full subcategory of $\cm_c$ which consists of the cofibrant objects which are $\lambda$-presentable objects in $\cm$. Every object $X \in C$ is a $\lambda$-filtered colimit of objects in $\ca$ and therefore also a split subobject of an object in $\ca$. We consider the canonical Quillen adjunction $F_C \colon \SSet^{\Delta C^{\op}} \rightleftarrows \cm \colon G_C$. Using (i), (iii) and similar arguments as before, it follows that the category $C \downarrow X$ is $\lambda$-filtered for every fibrant-cofibrant $X$ and the canonical morphisms 
$$\hocolim(C \downarrow X) \xrightarrow{\sim} \mathrm{colimit}(C \downarrow X) \xrightarrow{\cong} X$$
are weak equivalences. Assumption (b) of Proposition \ref{auxiliary-comparison} is satisfied, but we don't know about Assumption (a), so we will need a different argument to complete the proof. According to the proof of Proposition \ref{auxiliary-comparison} and Remark \ref{explaining-argument}, 
the canonical Quillen adjunction $(F_C, G_C)$ is a homotopy reflection if the canonical morphism 
$$\hocolim(\Delta C \downarrow X) \longrightarrow X$$
is a weak equivalence for every fibrant-cofibrant $X$. The functor $R$ from (ii) above defines a functor 
$$\rho \colon C \downarrow X \longrightarrow \Delta C \downarrow X$$
$$(E, E \to X) \mapsto (R(E), R(E)^0 \to E \to X).$$
There is also a functor in the other direction given by $(E^{\bullet}, E^0 \to X) \mapsto (E^0, E^0 \to X)$. For each $(E^{\bullet}, E^0 \to X) \in \Delta C \downarrow X$, we can find a lift in the following diagram  
\[
\xymatrix{
0 \ar[r] \ar[d] & R(E^0) \ar@{->>}[d]^{\sim} \\
E^{\bullet} \ar[r] \ar@{-->}[ru] & c E^0
}
\]
Since $C \downarrow X$ is filtered, it follows that the functor $\rho$ is (homotopy) cofinal. Therefore $\Delta C \downarrow X$ is also $\lambda$-filtered and, using (iii), there are canonical weak equivalences as follows
\[
\xymatrix{
\hocolim( \Delta C \downarrow X) \ar[d]^{\sim} \ar@<1ex>[r]_{\sim} & \hocolim(C \downarrow X) \ar[d]^{\sim} \ar@<1ex>[l] \\
\mathrm{colim}(\Delta C \downarrow X) \ar@<1ex>[r]_{\cong} & \mathrm{colim}(C \downarrow X) \ar@<1ex>[l].
}
\]
As a consequence, the Quillen adjunction $(F_C, G_C)$ is a homotopy reflection. The result then follows from Proposition \ref{localization}(2).
\end{proof}

Specializing to cofibrantly generated model categories, we obtain the following as immediate corollary. 

\begin{corollary} \label{main2}
Let $(\cm, \mathcal{C}of, \mathcal{W}, \mathcal{F}ib)$ be a cofibrantly generated model category. 
\begin{itemize}
\item[(1)] Suppose that $(\mathcal{C}of, \mathcal{W} \cap \mathcal{F}ib)$ is strictly cofibrantly generated. Assuming 
Vop\v{e}nka's principle, $\cm$ admits a small presentation. 
\item[(2)] Suppose that $(\mathcal{C}of, \mathcal{W} \cap \mathcal{F}ib)$ is $\aleph_0$-cofibrantly generated. Assuming 
Vop\v{e}nka's principle, $\cm$ admits a small presentation. 
\item[(3)] Suppose that $(\mathcal{C}of, \mathcal{W} \cap \mathcal{F}ib)$ is perfectly cofibrantly generated and 
$(\mathcal{C}of \cap \cw, \mathcal{F}ib)$ is semi-perfectly cofibrantly generated. Then $\cm$ admits a small presentation. 
\end{itemize}
In particular, in each of the cases (1)-(3), the model category $\cm$ is Quillen equivalent to a combinatorial model category. 
\end{corollary}

\begin{remark}
We do not know if the conclusion in Corollary \ref{main2} is true for all cofibrantly generated model categories. 

\end{remark}

\subsection{An Example}\label{counterexample}

\medskip

Theorem \ref{main} is false for model categories whose weak factorization systems are only weakly cofibrantly generated. This was shown in \cite{R2} assuming the 
negation of Vop\v{e}nka's principle, but the claim holds more generally under the standard set-theoretical assumptions. 

Let $\ck$ be the category of compact Hausdorff spaces with the trivial model structure where the weak equivalences are the isomorphisms. By Example \ref{compact-hausdorff}, 
this model structure has weakly cofibrantly generated weak factorization systems. It follows from the results in Subsections \ref{method} and \ref{results} that every combinatorial model category 
has a small homotopically dense subcategory. Since $\ck$ does not have a small (homotopically) dense subcategory, it is easy to conclude that it is not Quillen equivalent to a 
combinatorial model category.

Moreover, we claim that $\ck$ is not equivalent to the homotopy category of a combinatorial model category. To see this, we begin with the following 
property of these homotopy categories which will serve as a basic criterion (cf. \cite[Proposition 6.10]{R1}).

\begin{proposition} \label{criterion}
Let $\cm$ be a combinatorial model category and $\gamma: \cm \to \Ho(\cm)$ the localization functor. Then there is a regular cardinal $\lambda$ such that for 
each $\lambda$-presentable object $X \in \cm$, the object $\gamma(X)$ is $\lambda$-presentable with respect to coproducts in $\Ho(\cm)$, i.e. any morphism $\gamma(X) \to \coprod_{i\in I} Y_i$ in $\Ho(\cm)$ factors through a 
subobject $\coprod_{i \in J} Y_i$ with $|J|<\lambda$.
\end{proposition}
\begin{proof}
Let $\lambda$ be a regular cardinal such that the fibrant and cofibrant replacement functors in $\cm$ given by the small object argument preserve $\lambda$-filtered colimits and 
$\lambda$-presentable
objects. We denote these functors by $(-)^c$, $(-)^f$, and $(-)^{cf} : = (-)^f \circ (-)^c$. The existence of such $\lambda$ follows from \cite[Theorem 2.19]{AR}. 

Let $X$ be a $\lambda$-presentable object in $\cm$ and $f: \gamma(X) \to Y = \coprod^{\Ho(\cm)}_{i \in I} Y_i$ a morphism in $\Ho(\cm)$. Coproducts of cofibrant objects define coproducts 
in the homotopy category, hence $Y$ is canonically isomorphic to $\gamma(\coprod^{\cm}_{i \in I} Y_i^c)$. The associated morphism $f': \gamma(X) \to \gamma(\coprod^{\cm}_{i \in I} Y_i^c)$ 
can be represented, up to canonical isomorphism in $\Ho(\cm)$, by a morphism $\tilde{f} \colon X^{cf} \to (\coprod^{\cm}_{i \in I} Y_i^c)^f$. Since $(-)^f$ preserves $\lambda$-filtered 
colimits, the target of $\tilde{f}$ can be written as a $\lambda$-filtered colimit of $(\coprod^{\cm}_{j \in J} Y_j^c)^f$ for all $J \subseteq I$ with $|J| < \lambda$. Then since $X^{cf}$ is again 
$\lambda$-presentable in $\cm$, the morphism $\tilde{f}$ factors through $(\coprod^{\cm}_{j \in J} Y_j^c)^f$ for some $J \subseteq I$ with $|J| < \lambda$. This factorization induces the required 
factorization in the homotopy category. 
\end{proof}

Suppose that $\Ho(\ck)\cong\ck$ is equivalent to $\Ho(\cm)$ where $\cm$ is a combinatorial model category. Let $\gamma:\cm\to\Ho(\cm)$ be the localization functor to the 
homotopy category $\Ho(\cm)$. Since the property in Proposition \ref{criterion} is preserved by equivalences of categories, $\ck$ must contain many objects which are $\lambda$-presentable with respect to coproducts.
But $\emptyset$ is actually the only such object in $\ck$, hence the contradiction. Indeed, there are maps $*=\mathbb{F}(*) \to \mathbb{F}(I)$  which do not factor through any 
$\mathbb{F}(J)$ with $|J|<\lambda$. To see this, first recall that the inclusion $\mathbb{F}(J) \to \mathbb{F}(I)$ sends an ultrafilter $\cu$ on $J$ to the ultrafilter consisting of all subsets $X \subseteq I$ 
whose intersection $X\cap J$ belongs to $\cu$. Then take an ultrafilter on $I$ containing the filter of subsets whose complement has cardinality $<\lambda$. This defines a point in 
$\mathbb{F}(I)$ which does not factor through any $\mathbb{F}(J)$ with $|J|<\lambda$.

\section{Erratum}

Let $\cm$ be a cocomplete category and $S$ a set of objects. We denote $S$ also the full subcategory of $\cm$ spanned by $S$. An object $X \in \cm$ is called $S$-generated (see 
\cite{Ra}) if the canonical morphism from the canonical diagram of $X$ with respect to $S$, 
$$\kappa_S(X): = \mathrm{colim}(S \downarrow X \to \cm) \to X,$$
is an isomorphism. Let $\cm_S$ denote the full subcategory of $\cm$ spanned by the $S$-generated objects. It was claimed in \cite[Proposition 3.1(i)]{Ra} that the functor $\kappa_S: 
\cm \to \cm$ defines a right adjoint to the inclusion functor $\cm_S \to \cm$, but this claim is false in this generality. Indeed, the values of $\kappa_S$ are not $S$-generated objects 
in general. 

The main purpose of considering the full subcategory $\cm_S$ for the proof of the main result in \cite[Theorem 1.1]{Ra} was in order to obtain a coreflective subcategory of $\cm$ which contains $S$ and has a small dense subcategory, cf. \cite[Proposition 3.1(iii)]{Ra}. But the existence of such a coreflective category cannot be guaranteed in general, 
therefore also the proof of \cite[Theorem 1.1]{Ra} does not apply without additional assumptions. 

Here is a general construction of counterexamples to these claims.

\begin{example1}
Let $\cm$ be a cocomplete category and $S$ a set of objects such that the colimit-closure of $S$ in $\cm$ is $\cm$, obtained possibly after using iterated colimits, but $\cm$ does 
not have 
a small dense subcategory. Such examples exist, e.g., the category of compact Hausdorff spaces is the colimit-closure of the one-point space, but it does not have a small dense 
subcategory. Then $\cm_S$ cannot be $\cm$, because $\cm_S$ has a small dense subcategory, but then it cannot be closed under colimits in $\cm$. Therefore, $\cm_S$ is not coreflective 
in $\cm$. 
\end{example1}

Under the special assumption that such a coreflection exists, the proof of \cite[Theorem 1.1]{Ra} applies to give the following corrected statement. We say that a model
category $\cm$ is \emph{setwise cofibrantly generated} if both of its weak factorization systems are setwise cofibrantly generated.

\begin{theorem1}\label{correction}
Let $\cm$ be a setwise cofibrantly generated model category and let $S$ be the set of domains and codomains of the morphisms in the sets of generators for fibrations and trivial 
fibrations. Suppose that the inclusion functor $\cm_S \to \cm$ admits a right adjoint. Assuming Vop\v{e}nka's principle, then: 
\begin{itemize}
\item[(1)] $\cm_S$ is locally presentable.
\item[(2)] $\cm_S$ carries a cofibrantly generated model structure, inherited from $\cm$, and the inclusion functor $\cm_S \to \cm$ is a Quillen equivalence.
\end{itemize}
\end{theorem1}
\begin{proof}
The proof follows the proof of \cite[Theorem 1.1]{Ra}. $\cm_S$ is a cocomplete category with a small dense subcategory spanned by $S$. Assuming Vop\v{e}nka's principle, $\cm_S$ is locally presentable by 
\cite[Theorem 6.14]{AR}. 

The setwise cofibrantly generated model structure on $\cm$ restricts to a cofibrantly generated model structure on $\cm_S$ where the factorizations are given by the small object argument 
performed in $\cm_S$. 

Let $G: \cm \to 
\cm_S$ denote the right adjoint. It is easy to see that the counit morphism $G(X) \to X$ is a trivial fibration for all $X \in \cm$. In particular, $G$ preserves the weak equivalences. Then 
it follows easily that the inclusion $\cm_S \to \cm$ is a left Quillen equivalence.
\end{proof}

\begin{remark1}
If the generators of the weak factorization systems in $\cm$ permit the small object argument in $\cm_S$, then (2) holds without the assumption of Vop\v{e}nka's principle.
\end{remark1}

The special assumption of Theorem \ref{correction} is satisfied for topological fiber-small categories, see \cite{FR}. Moreover, in this case, 
Vop\v{e}nka's principle is not needed. It would be interesting to identify also other classes of model categories for which the special assumption of Theorem \ref{correction} is satisfied. 

\begin{corollary1}
Let $\cm$ be a setwise cofibrantly generated model category whose underlying category is a topological fiber-small category. Then $\cm$ is Quillen 
equivalent to a combinatorial model category. 
\end{corollary1}
\begin{proof}
Let $\ci$ be a generator for trivial fibrations and $\cj$ a generator for fibrations. Let $S$ be the set of objects that appear as domains 
or codomains of the morphisms in $\ci\cup\cj$. The full subcategory $\cm_S$ of $\cm$ consisting of $S$-generated objects is locally 
presentable and coreflective in $\cm$ \cite[Proposition 3.5 and Theorem 3.6]{FR}. Similarly to Theorem \ref{correction}, $\cm_S$ inherits a cofibrantly generated 
model structure from $\cm$ and the adjunction $\cm_S \rightleftarrows \cm$ is a Quillen equivalence.
\end{proof}

\end{document}